\documentclass[a4paper,reqno]{amsart}

\usepackage[pagebackref]{hyperref}
\usepackage{amsthm}
\usepackage{amsmath}
\usepackage{amsfonts}
\usepackage{amssymb}
\usepackage[T1]{fontenc}
\usepackage[arrow, matrix, curve]{xy}
\usepackage{picinpar, graphicx }
\usepackage{epic}

\newtheorem{Thm}{Theorem}
\newtheorem{Def}{Definition}
\newtheorem{Rem}{Remark}
\newtheorem{Lem}{Lemma}
\newtheorem{Prop}{Proposition}
\newtheorem{Cor}{Corollary}
\newtheorem{Q}{Question}
\newtheorem{Ex}{Example}

\newcommand{\Ww}{\mathcal{W}}
\newcommand{\Hh}{\mathcal{H}}
\newcommand{\Rr}{\mathcal{R}}
\newcommand{\Ee}{\mathcal{E}}
\newcommand{\Kk}{\mathcal{K}}
\newcommand{\RR}{\mathbb{R}}
\newcommand{\Gg}{\mathcal{G}}
\newcommand{\VV}{\mathbb{V}}
\newcommand{\Mm}{\mathcal{M}}

\begin{document}

\title[Local conservation laws by generalized isometric embeddings]{Construction of Local Conservation Laws by Generalized Isometric Embeddings of Vector Bundles}
\author{Nabil Kahouadji}

\date{April 10, 2008.}

\address{Institut de Math\'ematiques de Jussieu, Universit\'e Paris Diderot-Paris 7\\
UFR de Math\'ematiques\\
\'Equipe: G\'eom\'etrie et Dynamique\\
case 7012\\
2, place jussieu\\
75251-Paris Cedex 05, France.}
\email{\href{mailto:kahouadji@math.jussieu.fr}{kahouadji@math.jussieu.fr}}




\maketitle

\begin{abstract}
This article uses Cartan--K\"{a}hler theory to construct local conservation laws from covariantly closed vector valued differential forms, objects that can be given, for example, by harmonic maps between two Riemannian manifolds.  We apply the article's main result to construct conservation laws for covariant divergence free energy-momentum tensors. We also generalize the  local isometric embedding of surfaces in the analytic case by applying the main result to vector bundles of rank two over any surface.
\end{abstract}

\section{Introduction}

A conservation law can be seen as a  map defined on a space $\mathcal{F}$ (which can be for instance, a  function space, a fiber bundle section space, etc.) that associates  each  element $f$ of $\mathcal{F}$ with  a vector field $X$ on an $m$-dimensional Riemannian manifold $\mathcal{M}$, such that if $f$ is a solution to  a given PDE on $\mathcal{F}$, the vector field $X$ has a vanishing divergence. If we denote  by $g$ the Riemannian metric on the manifold  $\mathcal{M}$, we can canonically associate each vector field $X\in \Gamma(\mathrm{T}\mathcal{M})$ with  a differential 1-form $\alpha_{X}: = g(X,\cdot)$. Since $\mathrm{div}(X)= \ast \mathrm{d} \ast \alpha_{X}$  (or $\mathrm{div}(X)\mathrm{vol}_{\mathcal{M}}= \mathrm{d}(X \lrcorner \, \mathrm{vol}_{\mathcal{M}}))$, 
 where $\ast$ is the Hodge operator, $\mathrm{vol}_{\mathcal{M}}$ is the volume form on $\mathcal{M}$, and $X\lrcorner \, \mathrm{vol}_{\mathcal{M}}$ is the interior product of $\mathrm{vol}_{\mathcal{M}}$ by the vector field $X$, the requirement $\mathrm{div}(X)=0$ may be replaced by the requirement  $\mathrm{d}(X \lrcorner \, \mathrm{vol}_{\mathcal{M}}) = 0$, and hence, conservation laws may also be seen   as  maps on $\mathcal{F}$ with values on  differential $(m-1)$-forms such that solutions to PDEs are mapped to closed differential $(m-1)$-forms on $\mathcal{M}$. More generally, we could extend the notion of conservation laws as mapping to differential $p$-forms (for instance, Maxwell equations in vacuum can be expressed, as it is well-known, by requiring a system of differential 2-forms to be closed). In this paper, we address the question of finding conservation laws for a class of PDE described as follows: 
 
\begin{Q} 
Let   $\mathbb{V}$  be  an $n$-dimensional vector bundle over $\mathcal{M}$. Let $g$ be a metric bundle  and $\nabla$ a connection that is compatible with that metric. We then have a covariant derivative $\mathrm{d}_{\nabla}$ acting on vector valued differential forms. Assume that $\phi$ is a given covariantly closed  $\mathbb{V}$-valued differential $p$-form on $\mathcal{M}$, i.e.,
\begin{equation}\label{dnablaphi=0}
\mathrm{d}_{\nabla}\phi = 0.
\end{equation}
Does there exist $N\in \mathbb{N}$ and an embedding $\Psi$ of $\mathbb{V}$ into $\mathcal{M}\times \mathbb{R}^N$ given by $\Psi(x,X)=(x,\Psi_{x}X)$, where $\Psi_{x}$ is a linear map from $\mathbb{V}_{x}$ to $\mathbb{R}^N$ such that:
\begin{itemize}
\item $\Psi$ is isometric, i.e, for every $x\in \mathcal{M}$, the map $\Psi_{x}$ is an isometry,
\item if $\Psi(\phi)$ is the image of $\phi$ by $\Psi$, i.e.,  $\Psi(\phi)_{x} = \Psi_{x}\circ \phi_{x}$ for all $x\in \mathcal{M}$, then
\begin{equation}\label{dpsiphi=0}
\mathrm{d}\Psi(\phi)=0.
\end{equation}
\end{itemize}
\end{Q}

In this problem, the equation (\ref{dnablaphi=0}) represents  the given PDE (or a system of PDEs) and the map $\Psi$ plays the role of a conservation law. Note that the problem is trivial when the vector bundle is a line bundle. Indeed, the only connection on a real line bundle which is compatible with the metric is the flat one.\\


A fundamental example is the isometric embedding of Riemannian manifolds in Euclidean spaces and is related to the above problem as follows: $\mathcal{M}$ is an $m$-dimensional  Riemannian manifold, $\mathbb{V}$ is the tangent bundle $\mathrm{T}\mathcal{M}$, the connection $\nabla $ is the Levi-Civita connection, $p=1$  and the $\mathrm{T}\mathcal{M}$-valued  differential 1-form $\phi$ is the identity map on $\mathrm{T}\mathcal{M}$. Then  (\ref{dnablaphi=0}) expresses the torsion-free condition for the connection $\nabla$  and any solution $\Psi$ to (\ref{dpsiphi=0}) provides an isometric embedding $u$ of the Riemannian manifold $\mathcal{M}$ into a Euclidean space $\mathbb{R}^N$ through $\mathrm{d}u=\Psi(\phi)$, and conversely. An answer to the local analytic isometric embeddings of Riemannian manifolds is given by the Cartan--Janet theorem,  \cite{Cartan-Art,Janet-Art}. Despite the fact that  the Cartan--Janet result is local and the analicity hypothesis on the  data  may seem to be too restrictive, the Cartan--Janet theorem is important because it  actualizes the embedding in an optimal dimension  unlike the Nash--Moser isometric embedding which is a smooth and  global result. Consequently, if the above problem has a positive answer for $p=1$, the notion of isometric embeddings of Riemannian manifolds is extended to a notion  of \textit{generalized isometric embeddings} of vector bundles. The general problem, when $p$ is arbitrary, can also be viewed as an \textit{embedding of covariantly closed vector valued differential $p$-forms}.

Another  example expounded in \cite{Helein-Book} of such covariantly closed   vector valued differential forms is given by harmonic maps between two Riemannian manifolds. Indeed, let us consider a map $u$ defined on an $m$-dimensional Riemannian manifold $\mathcal{M}$ with values in an $n$-dimensional Riemannian manifold $\mathcal{N}$.  On the induced bundle\footnote{ or the pull back bundle.} by $u$ over $\mathcal{M}$, the $u^{\ast}\mathrm{T}\mathcal{N}$-valued differential  $(m-1)$-form $\ast \mathrm{d}u$ is covariantly closed  if and only if the map $u$ is harmonic, where the connection on the induced bundle is the pull back by $u$ of the Levi-Civita connection on $\mathcal{N}$.  A positive answer to the above problem in this case would make it possible to construct conservation laws on $\mathcal{M}$ from covariantly closed  vector valued differential $(m-1)$-forms, provided, for example, by harmonic maps. In his book \cite{Helein-Book}, motivated by the question of the compactness of weakly harmonic maps in Sobolev spaces in the weak topology (which is still an open question), H\'elein considers harmonic maps between Riemannian manifolds and explains how conservation laws may be obtained explicitly by the  Noether's theorem if the target manifold is symmetric and formulates the problem for non symmetric target manifolds.  
\\

In this article, our main result is a positive answer  when $p=m-1$ in the analytic case. We also find, as in the Cartan--Janet theorem,  the minimal required dimension that ensures the \textit{generalized isometric embedding} of an arbitrary vector bundle relative to a covariantly closed vector valued differential $(m-1)$-form.

\begin{Thm}\label{CLGIEVBThm}
Let $\mathbb{V}$ be a real analytic  $n$-dimensional vector bundle  over a real analytic $m$-dimensional manifold $\mathcal{M}$ endowed with a metric $g$ and a connection $\nabla$ compatible with $g$. Given a non-vanishing covariantly closed $\mathbb{V}$-valued differential $(m-1)$-form $\phi$, there exists a local isometric embedding of $\mathbb{V}$ in  $\displaystyle{\mathcal{M}\times \mathbb{R}^{n+\kappa^n_{m,m-1}}}$ over $\mathcal{M}$ where $\displaystyle{\kappa^n_{m,m-1} \geqslant (m-1)(n-1)}$ such that the image of $\phi$ is a conservation law.
\end{Thm}



The existence result Theorem \ref{CLGIEVBThm} can be applied to harmonic maps. We show in the last section of this paper a further application related to energy-momentum tensors which occur e.g. in general relativity. 
\\

The strategy for proving Theorem \ref{CLGIEVBThm} is the following: we reformulate the problem by means of an  exterior differential system on a manifold that must  be defined, and since all the data involved are real analytic, we use the Cartan--K\"{a}hler theory to prove the existence of integral manifolds. The problem can be represented by the following diagram that  summarizes the notations 

\begin{figure}[htbp]
\begin{center}
\begin{picture}(0,70)(0,-35)


\put(-45,30){\makebox(0,0){$\VV^n$}}
\put(-48,24){\vector(0,-1){50}}
\put(-46,-30){\makebox(0,0){$\tiny{\Mm^m}$}}
\put(-78,0){\makebox(0,0){$\tiny{(\mathrm{d}_{\nabla}\phi})_{p}\hspace{-1mm}=\hspace{-0.5mm}0$}}

\put(-75,30){\makebox(0,0){\small{$g,\nabla , $}}}

\qbezier(-32, 29)(-36,29)(-36, 31)
\qbezier(-36, 31)(-36,33)(-32, 33)
\drawline(-32,29)(13,29)
\qbezier(13,29)(7,30)(6,32)
\put(-13,35){\makebox(0,0){$\Psi$}}

\put(47,-30){\makebox(0,0){$\tiny{\Mm^m}$}}
\put(45,24){\vector(0,-1){50}}
\put(48,30){\makebox(0,0){$ \mathcal{M}^{m}\times\mathbb{R}^{N^n_{m,p}}$}}

\put(72,0){\makebox(0,0){$\tiny{\mathrm{d}\Psi(\phi)}\hspace{-1mm}=\hspace{-0.5mm}0$}}

\end{picture}
\caption{Generalized isometric embedding}
\label{default}
\end{center}
\end{figure}

\hspace{-4.5mm}where $N$ is an integer  that have to be defined in terms of the problem's  data: $n$, $m$ and $p$. Let us then set up a general strategy as an attempt to solve the general problem. We denote by $\kappa^n_{m,p}$ the  \textit{embedding codimension}, i.e., the dimension of the fiber extension in order to achieve the desired embedding. Since the Cartan--K\"{a}hler theory plays an important role in this paper and since the reader  may not be  familiar with exterior differential systems (EDS) and the Cartan--K\"{a}hler theorem, generalities are expounded  in section 2 concerning these notions and results. For details and proofs, the reader may consult \'Elie Cartan's book \cite{Cartan-Book} and the third chapter  of \cite{ExtDiffSys-Book}.

 Let us first recast our problem by using moving frames and coframes. For convenience, we adopt the following conventions for the indices: $i,j,k = 1 , \dots , n$ are the fiber indices , $\lambda , \mu , \nu = 1, \dots , m$ are the manifold indices and $a,b,c = n+1 , \dots , n+ \kappa^n_{m,p}$ are the extension indices.  We also adopt the Einstein summation convention, i.e., assume a summation when the same index is repeated in high and low positions. However, we will write the sign $\sum$ and make explicit the values of the summation indices when  necessary. Let $\eta=(\eta^1 , \dots , \eta^m)$ be a moving coframe on $\mathcal{M}$. Let $E=(E_{1}, \dots , E_{n})$ be an orthonormal  moving frame of $\mathbb{V}$. The covariantly closed  $\mathbb{V}$-valued differential $p$-form $\phi \in \Gamma(\wedge^p\mathcal{M}\otimes\mathbb{V}_{n})$ can be expressed as follows:
\begin{equation}
\phi=E_{i}\phi^{i}= E_{i}\psi^{i}_{\lambda_{1}, \dots , \lambda_{p}}\eta^{\lambda_{1}, \dots , \lambda_{p}}
\end{equation}
where $\psi^{i}_{\lambda_{1}, \dots , \lambda_{p}}$ are functions on $\mathcal{M}$. We assume that $1\leqslant \lambda_{1} < \dots < \lambda_{p} \leqslant m$ in the summation, and that  $\eta^{\lambda_{1}, \dots , \lambda_{p}}$ means $\eta^{\lambda_{1}}\wedge\dots\wedge\eta^{\lambda_{p}}$.

\begin{Def}
Let $\phi \in \Gamma(\wedge^p\mathcal{M}\otimes \mathbb{V})$  be a $\mathbb{V}$-valued differential $p$-form on $\mathcal{M}$. The generalized torsion of a connection relative to $\phi$ (or for short, a $\phi$-torsion) on a vector bundle over $\mathcal{M}$   is a $\mathbb{V}$-valued differential  $(p+1)$-form $\Theta=(\Theta^{i}):=\mathrm{d}_{\nabla}\phi$, i.e., in a local frame
\begin{equation}
\Theta = E_{i}\Theta^{i}:=E_{i}(\mathrm{d}\phi^{i}+ \eta^{i}_{j}\wedge\phi^{j}) \quad \text{ for all } \quad i \end{equation}
where $(\eta^{i}_{j})$ is the connection 1-form of $\nabla$ which is an $ \mathfrak{o}(n)$-valued differential 1-form (since  $\nabla$ is compatible with the metric bundle).
\end{Def}
Thus, the condition of being covariantly closed $\displaystyle{\mathrm{d}_{\nabla}\phi = 0}$ is equivalent to the fact that,  $\mathrm{d}\phi^{i}+\eta^{i}_{j}\wedge\phi^j = 0$ for all  $i = 1, \dots , n$. From the above definition, the connection $\nabla$ is $\phi$-torsion free. We also notice  that the generalized torsion defined above reduces to the standard torsion in the tangent bundle case when $\phi=E_{i}\psi^{i}_{\lambda}\eta^\lambda = E_{i}\eta^{i}$ (the functions  $\psi^{i}_{\lambda} = \delta^{i}_{\lambda}$ are the Kronecker tensors), and  the connection is Levi-Civita.
\begin{equation}
\mathrm{d}_{\nabla}\phi = 0 \Longleftrightarrow \mathrm{d}\phi^{i}+\eta^{i}_{j}\wedge\phi^j = 0 \quad \text{for all }  i = 1, \dots , n.
\end{equation}

Assume that the problem has a solution. We consider the flat connection 1-form $\omega$ on the Stiefel space $SO(n+\kappa^n_{m,p})/SO(\kappa^n_{m, p})$, the $n$-adapted frames of $\mathbb{R}^{(n+\kappa^n_{m,p})}$, i.e., the set of orthonormal families of $n$ vectors $\Upsilon = (e_{1}, \dots , e_{n})$ of $\mathbb{R}^{(n+\kappa^n_{m,p})}$ which can be completed by orthonormal $\kappa^n_{m,p}$ vectors $(e_{n+1}, \dots,e_{n+\kappa^n_{m,p}})$ to obtain an orthonormal set of $(n+\kappa^n_{m,p})$ vectors. Since we work locally, we will assume without loss of generality that we are given a cross-section $(e_{n+1},\dots ,e_{n+\kappa^n_{m,p}} )$ of the bundle fibration $SO(n+\kappa^n_{m,p})\longrightarrow SO(n+\kappa^n_{m,p})/SO(\kappa^n_{m,p})$. The flat standard 1-form of the connection $\omega$ is defined as follows: $\omega^{i}_{j} = \langle e_{i},\mathrm{d}e_{j}\rangle$ and $\omega^{a}_{i} = \langle e_{a}, \mathrm{d}e_{i}\rangle$, where $\langle,\rangle$ is the standard inner product on $\mathbb{R}^{n+\kappa^n_{m,p}}$. Notice that $\omega$ satisfies  Cartan's structure equations. Suppose now that such an isometric embedding exists, then,  if $e_{i}=\Psi(E_{i})$, the condition $\mathrm{d}\Psi(\phi)=0$ yields to
\begin{equation}
e_{i}(\mathrm{d}\phi^{i} + \omega^{i}_{j}\wedge\phi^j) + e_{a}(\omega^{a}_{i}\wedge\phi^{i})=0,
\end{equation}
a condition which is satisfied  if and only if
\begin{equation}
\eta^{i}_{j} = \Psi^{\ast}(\omega^{i}_{j}) \quad \text{ and }\quad \Psi^{\ast}(\omega^{a}_{i})\wedge\phi^{i} = 0.
\end{equation}

The problem then turns  to  finding moving frames $(e_{1}, \dots e_{n}, e_{n+1}, \dots , e_{n+\kappa^n_{m,p}})$ such that  there exist  $m$-dimensional integral manifolds of the exterior ideal generated by the naive exterior differential system $\{ \omega^{i}_{j} - \eta^{i}_{j} , \omega^{a}_{i}\wedge\phi^{i}\}$  on the product manifold
\begin{equation}
\mathbf{\Sigma}^n_{m,p} = \mathcal{M}\times \frac{SO(n+\kappa^n_{m,p})}{SO(\kappa^n_{m,p})}.
\end{equation}

Strictly speaking, the differential forms live in different spaces. Indeed, one should consider the projections $\pi_{\mathcal{M}}$ and $\pi_{St}$ of $\mathbf{\Sigma}^n_{m,p}$ on $\mathcal{M}$ and the Stiefel space and consider the ideal  on $\mathbf{\Sigma}^n_{m,p}$ generated by   $\pi^\ast_{\mathcal{M}}(\eta^{i}_{j}) - \pi^\ast_{St}(\omega^{i}_{j})$  and $\pi^\ast_{St}(\omega^{a}_{i})\wedge\pi^\ast_{\mathcal{M}}(\phi^{i})$. It seems reasonable however to simply write $\{\omega^{i}_{j}-\eta^{i}_{j} , \omega^{a}_{i}\wedge\phi^{i}\}$.\\

To find integral manifolds of the naive EDS, we  would need to check  that the exterior ideal is closed under the differentiation. However, this  turns out not to be the case. The idea is then to  add to the naive EDS the differential of the forms that generate it and therefore, we obtain a closed one. \\

The objects which we are dealing with in the following  have a geometric meaning in the tangent bundle case with a standard 1-form (the orthonormal moving coframe, as explained above) but not in the arbitrary vector bundle case as we noticed earlier with the notion of torsion of a connection. That leads us to define notions in a generalized sense in such a way that we recover the standard notions in the tangent bundle case. First of all, the Cartan lemma,  which in the isometric embedding problem implies the symmetry of the second fundamental form,  does not hold. Consequently, we can not assure nor assume that the coefficients of the second fundamental form are symmetric as in the isometric embedding problem. In fact, we will show that these conditions should be replaced by  \textit{generalized Cartan identities} that express how coefficients of the second fundamental form are related to each other, and of course, we recover the usual symmetry in the tangent bundle case. Another difficulty is the analogue of the  Bianchi identity of the curvature tensor. We will define  \textit{generalized Bianchi identities} relative to the covariantly closed  vector valued differential $p$-form and a \textit{generalized curvature tensor space} which corresponds, in the tangent bundle case, to the usual Bianchi identities and the Riemann curvature tensor space, respectively. Finally, besides the \textit{generalized Cartan identities} and \textit{generalized curvature tensor space}, we will make use of a  \textit{generalized Gauss map}. \\

The key to  the proof of Theorem \ref{CLGIEVBThm} is Lemma \ref{LemFond} for two main reasons: on one hand, it assures the existence of coefficients of the second fundamental form that satisfy the \textit{generalized Cartan identities} and the \textit{generalized Gauss equation}, properties that simplify the computation of the Cartan characters. On the other hand, the lemma gives the minimal required \textit{embedding codimension} $\kappa^n_{m,m-1}$ that ensures the desired embedding. Using Lemma \ref{LemFond}, we give  another proof of Theorem \ref{CLGIEVBThm} by an explicit construction of an ordinary integral flag. When the existence of integral manifold is established, we just need to project it on $\mathcal{M}\times\mathbb{R}^{n+\kappa^n_{m,p}}$.

\section{Generalities}

This section is a brief introduction to the Cartan--K\"{a}hler theory and is established to state the Cartan test, the Proposition \ref{PropCp} and the Cartan--K\"{a}hler theorem, results that we use in the proof of Theorem \ref{CLGIEVBThm}.

\subsection{EDS and exterior ideals}
Let us denote by $\mathcal{A}(\mathcal{M})$ the space of smooth differential forms on $\mathcal{M}$.\footnote{This is a graded algebra under the wedge product. We do not use the standard notation $\Omega (\mathcal{M})$ to not confuse it  with the curvature $2$-form of a connection.} An exterior differential system is a finite set of differential
forms $I=\{\omega_{1},\omega_{2},\dots , \omega_{k}\}\subset
\mathcal{A}(M)$ with which we associate the set of
equations $\{\omega_{i}=0 \,  | \, \omega_{i}\in I \}$.
 A subset of differential forms $\mathcal{I} \subset \mathcal{A}(\mathcal{M})$  is an exterior ideal if the exterior product of any differential form of $\mathcal{I}$ by a differential form of $\mathcal{A}(\mathcal{M})$ belongs to $\mathcal{I}$ and if the sum of any two differential forms of the
same degree belonging to $\mathcal{I}$, belongs also to $\mathcal{I}$.
The exterior ideal generated by $I$ is  the smallest exterior ideal
containing $I$. An exterior ideal is said to be   an exterior differential ideal if it is closed under the exterior differentiation and hence,  the
exterior differential ideal generated by an EDS is the smallest
exterior differential ideal containing the EDS. Let us notice that an EDS is closed if and only if the exterior differential ideal and the exterior ideal  generated by that EDS are  equal. In particular, if $I$ is an EDS,  $I\cup \mathrm{d}I$ is closed.

\subsection{Introduction to Cartan--K\"{a}hler theory}
Let $I\subset\mathcal{A}(\mathcal{M})$ be an EDS on $\mathcal{M}$ and
 let $\mathcal{N}$ be a submanifold of $\mathcal{M}$. The submanifold $\mathcal{N}$ is an integral manifold of  $I$ if  $\iota^{\ast}\varphi=0, \forall \varphi\in
I$, where $\iota$ is an embedding $\iota:\mathcal{N}\longrightarrow \mathcal{M}$. The purpose of this theory is to establish when a given EDS, which represents a PDE, has or does not have integral manifolds. We consider in this subsection, an $m$-dimensional real manifold $\mathcal{M}$ and $\mathcal{I}\subset \mathcal{A}(\mathcal{M})$ an exterior differential ideal on $\mathcal{M}$.

\begin{Def}Let $z\in \mathcal{M}$. A linear subspace  $E$ of $ T_{z}\mathcal{M}$ is an integral element of $\mathcal{I}$ if
$\varphi_{E}=0$ for all $\varphi \in \mathcal{I}$, where $\varphi_{E}$ means the evaluation of $\varphi$ on any basis of $E$. We denote by
$\mathcal{V}_{p}(\mathcal{I})$ the set of $p$-dimensional integral
elements of $\mathcal{I}$.
\end{Def}
$\mathcal{N}$ is an integral manifold of $\mathcal{I}$ if and only if each tangent space of  $\mathcal{N}$ is an integral element of $\mathcal{I}$. It is not hard to notice from the definition that a subspace of a given integral element is also an integral element.  We denote by $\mathcal{I}_{p}=\mathcal{I}\cap\mathcal{A}^{p}(\mathcal{M})$ the set of differential $p$-forms of $\mathcal{I}$. Thus, 
$\mathcal{V}_{p}(\mathcal{I})=\{E\in G_{p}(T\mathcal{M}) \, | \, \varphi_{E}=0$
for all $\varphi \in \mathcal{I}_{p} \}$.

\begin{Def}
Let $E$  be an integral element of $\mathcal{I}$. Let
$\{e_{1},e_{2},\dots ,e_{p}\}$ be a basis of  $E\subset T_{z}\mathcal{M}$. The
polar space of $E$, denoted by $H(E)$, is the vector space defined
as follows:
\begin{equation}
H(E)=\{v\in T_{z}\mathcal{M} \, |\,  \varphi(v, e_{1},e_{2},\dots , e_{p})=0
\text{ for all } \varphi \in \mathcal{I}_{p+1}\}.
\end{equation}
\end{Def}

Notice that $E\subset H(E)$.  The polar space plays an
important role in exterior differential system theory as we can notice from  the following proposition.

\begin{Prop}
Let $E \in \mathcal{V}_{p}(\mathcal{I})$ be a $p$-dimensional
integral element of
 $\mathcal{I}$. A $(p+1)$-dimensional vector space $E^{+}\subset T_{z}\mathcal{M}$ which contains $E$
 is an integral element of  $\mathcal{I}$ if and only if $E^{+}\subset H(E)$.
\end{Prop}

In order to check if a given $p$-dimensional integral element of an
EDS $\mathcal{I}$ is contained in a
$(p+1)$-dimensional integral element of  $\mathcal{I}$ , we
introduce the following function called the extension rank $
r:\mathcal{V}_{p}(\mathcal{I}) \longrightarrow
\mathbb{Z}$ that associates  each  integral element $E$ with an integer 
$r(E)=dim H(E)-(p+1)$. The extension rank  $r(E)$ is in fact the dimension of $\mathbb{P}(H(E)/E)$ and is always greater than $-1$. If $r(E)=-1$, then  $E$ is contained in any $(p+1)$-dimensional integral element of $\mathcal{I}$ and consequently, there is no hope of extending the integral element. An integral element E is said to be regular if $r(E)$ is constant on a  neighborhood of $z$.

\begin{Def}
An integral flag of  $\mathcal{I}$ on $z\in \mathcal{M}$ of length $n$ is a
sequence $(0)_{z}\subset E_{1}\subset E_{2}\subset \dots \subset
E_{n}\subset T_{z}\mathcal{M}$ of integral elements $E_{k}$ of $\mathcal{I}$.
\end{Def}
An integral element $E$ is said to be ordinary
if its base point $z\in \mathcal{M}$ is an ordinary zero of $I_{0}=I\cap
\mathcal{A}^{0}(\mathcal{M})$ and if there exists an integral flag
$(0)_{z}\subset E_{1}\subset E_{2}\subset \dots \subset E_{n}=E
\subset T_{z}\mathcal{M}$ where the $E_{k}$, $k=1,\dots ,(n-1)$ are
regular. Moreover, if  $E_{n}$ is itself
regular, then  $E$ is said  to be regular. We can now state the following important results of the Cartan--K\"{a}hler theory.

\begin{Thm}\label{TestCartan}(Cartan's test)
Let  $\mathcal{I}\subset \mathcal{A}^{\ast}(\mathcal{M})$ be  an exterior ideal
which does not contain  0-forms (functions on $\mathcal{M}$). Let  $(0)_{z}\subset E_{1}\subset
E_{2}\subset \dots \subset E_{n}\subset T_{z}\mathcal{M}$ be an integral flag
of $\mathcal{I}$. For any $k<n$, we denote by $C_{k}$ the
codimension of the polar space $H(E_{k})$ in $T_{z}\mathcal{M}$. Then
$\mathcal{V}_{n}(\mathcal{I})\subset G_{n}(T\mathcal{M})$ is at least of
codimension $C_{0}+C_{1}+\dots + C_{n-1}$ at $E_{n}$.
Moreover, $E_{n}$ is an ordinary integral flag if and only if
$E_{n}$ has a neighborhood $U$ in $G_{n}(T\mathcal{M})$ such that
$\mathcal{V}_{n}(\mathcal{I})\cap U$ is a manifold of
codimension $C_{0}+C_{1}+\dots + C_{n-1}$  in  $U$.
\end{Thm}

The numbers $C_{k}$ are called Cartan characters of the $k$-integral element. The following proposition is useful in the applications. It allows us to compute the Cartan characters of the constructed flag in the proof of the Theorem \ref{CLGIEVBThm}.

\begin{Prop}\label{PropCp}

At a point $z\in \mathcal{M}$,  let $E$ be an $n$-dimensional integral element of an exterior ideal $\mathcal{I}\cap \mathcal{A}^{\ast}(\mathcal{M})$ which does not contain differential 0-forms.  Let $\omega_{1},\omega_{2},\dots ,
\omega_{n},\pi_{1},\pi_{2}, \dots , \pi_{s}$ (where $s=dim\, \mathcal{M} - n$) be a coframe in a open neighborhood of  $z\in M$ such that $E=\{v
\in T_{z}\mathcal{M}\,  | \, \pi_{a}(v)=0 \text{ for all } a=1,\dots ,  s\}$. For
all  $p\leqslant n$, we define $E_{p}=\{v\in E \, | \, \omega_{k}(v)=0 \text{
for all } k > p \}$. Let $\{\varphi_{1}, \varphi_{2},\dots
,\varphi_{r}\}$ be the set of differential forms which generate the
exterior ideal  $\mathcal{I}$, where $\varphi_{\rho}$ is of degree $(d_{\rho}+1)$.
For all $\rho$, there exists an expansion
\begin{equation}
\varphi_{\rho}=\sum_{|J|=d_{\rho}}\pi_{\rho}^{J}\wedge\omega_{J}+\tilde{\varphi}_{\rho}
\end{equation}
where the   1-forms $\pi_{\rho}^{J}$ are linear combinations of
the forms $\pi$ and the terms $\tilde{\varphi}_{\rho}$ are, either
of degree  2 or more on  $\pi$, or vanish at $z$.
Moreover, we have
\begin{equation}
H(E_{p})=\{v\in T_{z}\mathcal{M} | \pi_{\rho}^{J}(v)=0 \text{ for all } \rho
\text{ and } \sup J\leqslant p\}
\end{equation}
In particular, for the integral flag $(0)_{z}\subset E_{1}\subset
E_{2}\subset\dots \subset E_{n}\cap T_{z}M$ of $\mathcal{I}$,
the Cartan characters $C_{p}$ correspond to the number of linear independent  forms
$\{\pi_{\rho}^{J}|_{z} \text{ such that } \sup J \leqslant p\}.$
\end{Prop}

The following theorem is of  great importance not only because it is a  generalization of the well-known Frobenius theorem but also because it represents a generalization of the  Cauchy--Kovalevskaya theorem.

\begin{Thm}\label{ThCartanKahler}(Cartan--K\"{a}hler)\\
Let $\mathcal{I}\subset \mathcal{A}^{\ast}(\mathcal{M})$ be a real analytic
exterior differential ideal. Let $\mathcal{X} \subset \mathcal{M}$ be a $p$-dimensional connected real analytic K\"{a}hler-regular integral manifold of $\mathcal{I}$.
Suppose that $r=r(\mathcal{X})\geqslant 0$. Let  $\mathcal{Z}\subset \mathcal{M}$ be a real analytic
submanifold of $\mathcal{M}$ of  codimension $r$ which contains  $\mathcal{X}$ and such
that  $T_{x}\mathcal{Z}$ and
$H(T_{x}\mathcal{X})$ are transverse in  $T_{x}\mathcal{M}$ for all  $x\in \mathcal{X} \subset \mathcal{M}$.
There exists then a $(p+1)$-dimensional connected real
analytic integral manifold $\mathcal{Y}$ of $\mathcal{I}$, such that
$\mathcal{X}\subset \mathcal{Y} \subset \mathcal{Z}$. Moreover, $\mathcal{Y}$
is unique in the sense that another integral manifold of
$\mathcal{I}$ having the stated properties, coincides
with $\mathcal{Y}$ on an open neighborhood of  $\mathcal{X}$.
\end{Thm}

The analycity condition of the exterior differential ideal is crucial because of the requirements in the Cauchy--Kovalevskaya theorem used in the Cartan--K\"{a}hler theorem's proof. It has an important corollary. Actually, in the application, this corollary is more often  used than the theorem and is sometimes called  \textit{the Cartan--K\"{a}hler theorem} in literature.

\begin{Cor}\label{CorCartanKahler}(Cartan--K\"{a}hler)\\
Let $\mathcal{I}$ be an analytic exterior differential ideal on a manifold
$\mathcal{M}$. If $E\subset T_{z}M$ is an ordinary integral element of
$\mathcal{I}$, there exists an integral manifold of
$\mathcal{I}$ passing through $z$ and having $E$ as a tangent
space at  point $z$.
\end{Cor}

One of  the great applications of the Cartan--K\"{a}hler theory is the Cartan--Janet theorem concerning  the local isometric embedding of Riemannian manifolds. We mention this theorem for its historical importance and  because the result of this paper generalizes isometric embedding of surfaces.

\begin{Thm}\label{ThBCJS}(Cartan \cite{Cartan-Art}--Janet \cite{Janet-Art})\\
Every $m$-dimensional real analytic Riemannian manifold can be
locally embedded isometrically in an
$\displaystyle{\frac{m(m+1)}{2}}$-dimensional Euclidean space.
\end{Thm}

\section{Construction of Conservation Laws}

In this section, we continue to explain   the general strategy of solving the problem in the general case started in the introduction,  and  we give a complete proof of Theorem \ref{CLGIEVBThm}. In the introduction, we showed that solving the general problem is equivalent to looking  for the existence of integral manifolds of the naive EDS $\{ \omega^{i}_{j} - \eta^{i}_{j} , \omega^{a}_{i}\wedge\phi^{i} \}$ on the product manifold 
\begin{equation}\nonumber
\mathbf{\Sigma}^n_{m,p} = \mathcal{M}_{m}\times \frac{SO(n+\kappa^n_{m,p})}{SO(\kappa^n_{m,p})}.
\end{equation}

This naive EDS is not closed. Indeed, the  generalized torsion-free of the connection implies that $\mathrm{d}(\omega^{a}_{i}\wedge\phi^{i}) \equiv 0 $ modulo the naive EDS, but the Cartan's second-structure equation yields to  $\displaystyle{\mathrm{d}(\omega^{i}_{j} - \eta^{i}_{j}) \equiv \sum_{a}(\omega^{a}_{i}\wedge\omega^{a}_{j}) - \Omega^{i}_{j}}$ modulo the naive EDS, where $\Omega = (\Omega^{i}_{j})$ is the curvature 2-form of the connection. Consequently, the exterior ideal that we now consider on  the product manifold $\mathbf{\Sigma}^n_{m,p}$ is 
\begin{equation}
\mathcal{I}^n_{m,p} = \{ \omega^{i}_{j} - \eta^{i}_{j} ,  \sum_{a}\omega^{a}_{i}\wedge\omega^{a}_{j} - \Omega^{i}_{j} ,  \omega^{a}_{i}\wedge\phi^{i} \}_{\text{alg}}.
\end{equation}

The curvature 2-form of the connection is an $\mathfrak{o}(n)$-valued two form and is related to the connection 1-form $(\eta^{i}_{j})$ by the Cartan's second-structure equation:
\begin{equation}\label{CartanStructureEq}
\Omega^{i}_{j}= d\eta^{i}_{j} + \eta^{i}_{k}\wedge\eta^{k}_{j}
\end{equation}
A first covariant derivative of $\phi$ has led to the generalized torsion. A second covariant derivative of $\phi$ gives rise to  \textit{generalized Bianchi identities}\footnote{In the tangent bundle case and $\phi = E_{i}\eta^{i}$, we recover the standard Bianchi identities of the Riemann curvature tensor, i.e, $\mathcal{R}^{i}_{jkl} =\mathcal{R}^{k}_{lij}$ and  $\mathcal{R}^{i}_{jkl} + \mathcal{R}^{i}_{ljk}  + \mathcal{R}^{i}_{klj}  = 0 $ .} as follows:
\begin{equation}
\mathrm{d}_{\nabla}^2(\phi) = 0 \Longleftrightarrow \Omega^{i}_{j}\wedge\phi^{j} = 0 \text{ for all } i = 1, \dots , n.
\end{equation}
The conditions  $\Omega^{i}_{j}\wedge \phi^{i}=0$ for all $i = 1, \dots , n$ are called \textit{generalized Bianchi identities}. We then define  a generalized curvature tensor space $\mathcal{K}^n_{m,p}$ as the space of curvature tensor satisfying the \textit{generalized Bianchi identities}:
\begin{equation}
\mathcal{K}^n_{m,p}=\{ (\mathcal{R}^{i}_{j;\lambda \mu}) \in \wedge^2(\mathbb{R}^n)\otimes\wedge^2(\mathbb{R}^m) | \, \Omega^{i}_{j}\wedge\phi^{i} = 0\}
\end{equation}
where $\Omega^{i}_{j} = \frac{1}{2}\mathcal{R}^{i}_{j;\lambda\mu}\eta^\lambda\wedge\eta^\mu = \mathcal{R}^{i}_{j;\lambda\mu}\eta^\lambda\otimes\eta^\mu$.

In the tangent bundle case and $\phi = E_{i}\eta^{i}$, $\mathcal{K}^n_{n,1}$ is the Riemann curvature tensor space which is of dimension $\displaystyle{\frac{1}{12}m^2(m^2-1)}$. 
\\


All the data are analytic, we can apply the Cartan--K\"{a}hler theory if we are able to check the involution of  the exterior differential system by constructing an $m$-integral flag: If the exterior ideal $\mathcal{I}^n_{m,p}$ passes the Cartan test, the flag is then ordinary and by the Cartan--K\"{a}hler theorem, there exist integral manifolds of $\mathcal{I}^n_{m,p}$. To be able to project the product manifold $\mathbf{\Sigma}^n_{m,p}$ on $\mathcal{M}$, we also need to show  the existence of $m$-dimensional integral manifolds on which the volume form on $\eta^{1,\dots , m}$ on  $\mathcal{M}$ does not vanish. 
\\ 
 
The EDS is not involutive and hence we ''prolong'' it by introducing new variables. Let us express the  $1$-forms $\omega^{a}_{i}$ in  the coframe $(\eta^1, \dots , \eta^m)$ in order to later make  the computation of Cartan characters easier . Let $\mathcal{W}^n_{m,p}$ be an $\kappa^n_{m,p}$-dimensional Euclidean space. We then write $\omega^{a}_{i} = H^a_{i\lambda}\eta^{\lambda}$ where $H^{a}_{i\lambda}\in \mathcal{W}^n_{m,m-1}\otimes \mathbb{R}^n\otimes\mathbb{R}^m$ and define the forms $\pi^{a}_{i} = \omega^{a}_{i}  - H^a_{i\lambda}\eta^{\lambda}$. We can also consider $H_{i\lambda}= (H^{a}_{i\lambda})$ as a vector of $\mathcal{W}^{n}_{m,p}$. The forms that generate algebraically $\mathcal{I}^n_{m,p}$ are then expressed as follows:
\begin{equation}\label{GaussGenExpression}
\begin{split}
\sum_{a}\omega^{a}_{i}\wedge\omega^{a}_{j}-\Omega^{i}_{j}& = \sum_{a}\pi^{a}_{i}\wedge\pi^{a}_{i} +\sum_{a} (H^{a}_{j\lambda} \pi^{a}_{i}- H^{a}_{i\lambda}\pi^{a}_{j})\wedge\eta^\lambda \\ & +\frac{1}{2} \sum_{a}\underbrace{(H^{a}_{i\lambda}H^{a}_{j\mu} - H^{a}_{i\mu}H^{a}_{j\lambda} - \mathcal{R}^{i}_{j;\lambda\mu})}_{*}\eta^\lambda\wedge\eta^\mu
\end{split}
\end{equation}
and
\begin{equation}\label{CartanGenExpression}
\omega^{a}_{i}\wedge\phi^{i}= \psi^{i}_{\lambda_{1}\dots\lambda_{p}}\pi^{a}_{i}\wedge\eta^{\lambda_{1}\dots \lambda_{p}} + \hspace{-1.2cm}\sum_{\tiny{\begin{array}{c}\lambda = 1,\dots , m \\1\leqslant \mu_{1}<\dots <\mu_{p}\leqslant m\end{array}}}\hspace{-1.2cm}\overbrace{H^{a}_{i\lambda}\psi^{i}_{\mu_{1}, \dots , \mu_{p}}}^{**}\eta^{\lambda\mu_{1}\dots\mu_{p}}.
\end{equation}

These new expressions of the forms in terms of vectors $H$ and the differential $1$-form $\pi$ will help us  compute the Cartan characters of an $m$-integral flag. To simplify these calculations, we will choose $H^{a}_{i\lambda}$, which are the coefficients of the second fundamental form, so that the quantities marked with $(*)$ and $(**)$ in the equations (\ref{GaussGenExpression}) and (\ref{CartanGenExpression}) vanish, and hence:
\begin{eqnarray}
\label{GenGaussEq}\sum_{a}(H^{a}_{i\lambda}H^{a}_{j\mu} - H^{a}_{i\mu}H^{a}_{j\lambda} )= \mathcal{R}^{i}_{j;\lambda\mu} \qquad \text{generalized Gauss equation}\\
\label{GenCartanId}\sum_{\tiny{\begin{array}{c}\lambda = 1,\dots , m \\1\leqslant \mu_{1}<\dots <\mu_{p}\leqslant m\end{array}}}\hspace{-1.2cm}H^{a}_{i\lambda}\psi^{i}_{\mu_{1}, \dots , \mu_{p}}\eta^{\lambda\mu_{1}\dots\mu_{p}} =0\qquad \text{generalized Cartan identities}.
\end{eqnarray}

As we mentioned in the introduction, the system of equations (\ref{GenCartanId}) is said to be \textit{generalized Cartan identities} because it gives us  relations between the coefficients of the second fundamental form which are not necessarily the usual symmetry given by the Cartan lemma. These properties of the coefficients and the fact that  the curvature tensor $(\mathcal{R}^{i}_{j;\lambda\mu})$ satisfies generalized Bianchi identities yield us to name the equation (\ref{GenGaussEq}) as the  \textit{generalized Gauss equation}. 

We now  define a \textit{generalized Gauss map} $\mathcal{G}^n_{m,p}:\mathcal{W}^n_{m,p}\otimes \mathbb{R}^n\otimes\mathbb{R}^m \longrightarrow \mathcal{K}^n_{m,p}$ defined for $H^{a}_{i\lambda}\in \mathcal{W}^n_{m,p}\otimes \mathbb{R}^n\otimes\mathbb{R}^m $ by 
\begin{equation}
\Big(\mathcal{G}^n_{m,p}(H)\Big)^{i}_{j;\lambda\mu}= \sum_{a}(H^{a}_{i\lambda}H^{a}_{j\mu} - H^{a}_{i\mu}H^{a}_{j\lambda}).
\end{equation}

Let us specialize in the conservation laws case, i.e., when $p=m-1$. We adopt the following notations: $\Lambda=(1,2,\dots, m)$ and $\Lambda\smallsetminus k = (1,\dots , k-1 ,k+1 ,\dots , m)$. We thus have  $\eta^{\Lambda} = \eta^1\wedge\dots\wedge\eta^m$ and $\eta^{\Lambda \smallsetminus k}= \eta^1\wedge\dots\wedge\eta^{k-1}\wedge\eta^{k+1}\dots\wedge\eta^m$. Let us construct an ordinary $m$-dimensional integral element of the exterior ideal $\mathcal{I}^n_{m,m-1}$ on $\mathbf{\Sigma}^n_{m,m-1}$. Generalized Bianchi identities are trivial in this case and so $\displaystyle{\text{dim}\,\mathcal{K}^n_{m,m-1} = \frac{n(n-1)}{2}\frac{m(m-1)}{2}}$.

The generalized Gauss equation is $H_{i\lambda}.H_{j\mu} - H_{i\mu}.H_{j\lambda} = \mathcal{R}^{i}_{j;\lambda\mu}$, where $H_{i\lambda}$ is viewed as a vector of the $\kappa^n_{m,m-1}$-Euclidean space $\mathcal{W}^{n}_{m,m-1}$. Generalized Cartan identities are
\begin{equation}
\sum_{\lambda = 1,\dots , m}(-1)^{\lambda+1} H^{a}_{i\lambda}\psi^{i}_{\Lambda \smallsetminus \lambda} = 0 \qquad \text{for all } a.
\end{equation}

The following lemma, for which a proof is later given, represents the key to the proof of Theorem \ref{CLGIEVBThm}.

\begin{Lem}\label{LemFond} Let $\kappa^n_{m,m-1}\geqslant (m-1)(n-1)$ and $\mathcal{W}^n_{m,m-1}$ be a Euclidean space of dimension $\kappa^n_{m,m-1}$. Let $\mathcal{H}^n_{m,m-1} \subset \mathcal{W}^n_{m,m-1}\otimes \mathbb{R}^n\otimes\mathbb{R}^m $ be the open set consisting of those elements $H=(H^a_{i\lambda})$ so that the vectors $\{ H_{i\lambda} | i=1,\dots , n-1  \text{ and } \lambda = 1, \dots , m-1\}$ are linearly independents as elements of $\mathcal{W}^n_{m,m-1}$ and satisfy  generalized Cartan identities. Then 
$ \mathcal{G}^n_{m,m-1} : \mathcal{H}^n_{m,m-1} \longrightarrow \mathcal{K}^n_{m,m-1}$ is a surjective submersion.
\end{Lem}

Let $\mathcal{Z}^n_{m,m-1}=\{ (M, \Upsilon, H) \in \mathbf{\Sigma}^n_{m,m-1}\times \Ww^n_{m,m-1}\otimes\RR^n\otimes\RR^m \, |\, H\in \mathcal{H}^n_{m,m-1} \}$. We conclude from Lemma \ref{LemFond} that $\mathcal{Z}^n_{m,m-1}$ is a submanifold\footnote{$\mathcal{Z}^n_{m,m-1}$ is the fiber of $\mathcal{R}$ by a submersion. The surjectivity of $\mathcal{G}^n_{m,m-1}$ assures the non-emptiness.} and hence, 
\begin{equation}
\dim\, \mathcal{Z}^n_{m,m-1} = \dim \, \mathbf{\Sigma}^n_{m,m-1} + \dim \, \mathcal{H}^n_{m,m-1} 
\end{equation}
where
\begin{eqnarray}
\dim \, \mathbf{\Sigma}^n_{m,m-1} = m+\frac{n(n-1)}{2}+n\kappa^n_{m,m-1}\\
\dim\, \mathcal{H}^n_{m,m-1} =(nm-1)\kappa^n_{m,m-1} - \frac{n(n-1)m(m-1)}{4}.
\end{eqnarray}

We define the map $\Phi^n_{m,m-1}:\mathcal{Z}^n_{m,m-1}\longrightarrow \mathcal{V}_{m}(\mathcal{I}^n_{m,m-1} , \eta^\Lambda)$ which associates  $(x,\Upsilon , H)\in \mathcal{Z}^n_{m,m-1}$ with the $m$-plan on which the differential forms that generate algebraically $\mathcal{I}^n_{m,m-1}$ vanish and the volume form $\eta^\Lambda$ on $\mathcal{M}$ does not vanish. $\Phi^n_{m,m-1}$ is then an embedding and hence $\dim\, \Phi(\mathcal{Z}^n_{m,m-1})= \dim\, \mathcal{Z}^n_{m,m-1}$. In what follows, we prove that in fact $\Phi(\mathcal{Z}^n_{m,m-1})$ contains only ordinary $m$-integral elements of $\mathcal{I}^n_{m,m-1}$. Since the coefficients $H^a_{i\lambda}$ satisfy the generalized Gauss equation and generalized Cartan identities, the differential  forms that generate the exterior ideal $\mathcal{I}^n_{m,m-1}$ are as follows:
\begin{eqnarray}
\sum_{a} \omega^{a}_{i}\wedge\omega^{a}_{j}-\Omega^{i}_{j} = \sum_{a}\pi^{a}_{i}\wedge\pi^{a}_{i} +\sum_{a} (H^{a}_{j\lambda} \pi^{a}_{i}- H^{a}_{i\lambda}\pi^{a}_{j})\wedge\eta^\lambda \\
\omega^{a}_{i}\wedge\phi^{i}= \psi^{i}_{\lambda_{1}\dots\lambda_{p}}\pi^{a}_{i}\wedge\eta^{\lambda_{1}\dots \lambda_{p}}.
\end{eqnarray}
We recall that  Cartan characters are the codimension of the polar space of integral elements. Their computations are a straightforward application of Proposition \ref{PropCp}. and yield
\begin{eqnarray}
C_{\lambda}= \frac{n(n-1)}{2}(\lambda +1)\text{ for }\lambda=0,..m-2 \\ C_{m-1} = \frac{n(n-1)}{2}m + \kappa^n_{m,m-1}
\end{eqnarray}
so 
\begin{equation}
C_{0}+\dots + C_{m-1}= m\frac{n(n-1)}{2}+\frac{n(n-1)m(m-1)}{4}+ \kappa^n_{m,m-1}.
\end{equation}
Finally, the codimension of the space on $m$-integral elements of $\mathcal{I}^{n}_{m,m-1}$ on which $\eta^\Lambda$ does not vanish is:
\begin{equation}
\begin{split}
\mathrm{codim}\, \mathcal{V}_{m}(\mathcal{I}^n_{m,m-1},\eta^\Lambda) & = \mathrm{dim}\, G_{m}\Big(T_{(x,\Upsilon)}\mathbf{\Sigma}^n_{m,m-1} \Big)- \Phi(\mathcal{Z}^n_{m,m-1})\\
&=m\frac{n(n-1)}{2}+\frac{n(n-1)m(m-1)}{4}+ \kappa^n_{m,m-1}.
\end{split}
\end{equation}
By the Cartan test, we conclude that $\Phi(\mathcal{Z}^n_{m,m-1})$ contains only ordinary $m$-integral flags. The Cartan--K\"{a}hler theorem then assures the existence of an $m$-integral manifold on which $\eta^\Lambda$ does not  vanish since the exterior ideal is in involution. We finally project the integral manifold on $\mathcal{M}\times \mathbb{R}^{n+\kappa}$. Let us notice that the requirement of the non vanishing of  the volume form  $\eta^\Lambda$ on the integral manifold yields to project the integral manifold on $\mathcal{M}$ and also to view it as a graph of a function $f$ defined on $\mathcal{M}$ with values in the space of $n$-adapted orthonormal frames of $\mathbb{R}^{n+\kappa}$. In  the isometric embedding problem, the composition of $f$ with the projection of the frames on the Euclidean space is by construction the  isometric embedding map.

\subsection{Another proof of Theorem \ref{CLGIEVBThm}}
This proof is based on explicitly constructing  an ordinary $m$-integral element, and the Cartan characters are computed by expliciting the polar space of an integral flag. As defined above, let us consider $\mathcal{I}^n_{m,m-1}$ an exterior ideal on $\mathbf{\Sigma}^n_{m,m-1}$.
Let us denote by  $(X_{\lambda})$ the dual basis of $(\eta^\lambda)$ and by $(Y_{A})$ the dual basis of $(\varpi^A)=(\varpi^{\sigma(^{i}_{j})} , \varpi^{\sigma(^a_{j})})=(\omega^{i}_{j}-\eta^{i}_{j} , \omega^{a}_{i})$ where $A= 1, \dots , \mathrm{dim} \, \mathbf{\Sigma}^n_{m,m-1} - m$ and 
$\displaystyle{\sigma(^{i}_{j}) = (j-i) + \frac{n(n-1)}{2} - \frac{(n-i)(n-i+1)}{2}}$ for $1\leqslant i < j \leqslant n $ and $\displaystyle{\sigma(^{a}_{i})=\frac{n(n-1)}{2} + (a-n-1)n + i}$ for $i=1,\dots , n$ and $a=n+1 , \dots , n+\kappa^n_{m,m-1}$. Let us consider on the Grassmannian manifold $G_{m}(\mathbf{\Sigma}^n_{m,m-1}, \eta^\Lambda)$ a basis $\mathfrak{X}_{\lambda}$ defined as follows:
\begin{equation}
\mathfrak{X}_{\lambda}(E)=X_\lambda + P^{A}_{\lambda}(E)Y_{A}\qquad A=1,\dots , \mathrm{dim}\, \mathbf{\Sigma}^n_{m,m-1} - m.
\end{equation} 

Let $(\Pi^\lambda (E))$ be the dual basis of $(\mathfrak{X}_{\lambda}(E))$. In order to compute the codimension in the Grassmannian $G_{m}(T\mathbf{\Sigma}^n_{m,m-1}, \eta^\Lambda)$ of $m$-integral elements of $\mathcal{I}^n_{m,m-1}$, we pull back the forms that generate the exterior ideal. To do so, we evaluate the forms on the basis $\mathfrak{X}_{\lambda}(E)$ and hence the expression of the forms on the Grassmannian are:  
\begin{eqnarray*}
&\displaystyle{(\varpi^{\sigma(^{i}_{j})})_{E} = P^{\sigma(^{i}_{j})}_{\lambda}\Pi^{\lambda}}\\
&\displaystyle{\Big( \sum_{a} \varpi^{\sigma(^{a}_{i})}\wedge \varpi^{\sigma(^{a}_{j})} -\Omega^{i}_{j} \Big)_{E} = \Big(\sum_{a}P^{\sigma(^{a}_{i})}_{\lambda}P^{\sigma(^{a}_{j})}_{\mu} - P^{\sigma(^{a}_{i})}_{\mu}P^{\sigma(^{a}_{j})}_{\lambda} - \mathcal{R}^{i}_{j;\lambda\mu}\Big)\Pi^{\lambda\mu}}\\
&\displaystyle{(\varpi^{\sigma^{a}_{j}}\wedge\phi^{i})_{E} = \Big(\sum_{\lambda}(-1)^{\lambda + 1}\psi^{i}_{\Lambda \smallsetminus \lambda}P^{\sigma(^{a}_{i})}_{\lambda}\Big)\Pi^{\Lambda} }.
\end{eqnarray*}
The number of functions that have linearly independent differentials represents the desired codimension, and hence with lemma \ref{LemFond}:
\begin{equation}
\mathrm{codim}\, \mathcal{V}_{m}(\mathcal{I}^n_{m,m-1}, \eta^\Lambda )= m\frac{n(n-1)}{2}+\frac{n(n-1)}{2}\frac{m(m-1)}{2} + \kappa^n_{m,m-1}
\end{equation}
Let us now construct an explicit $m$-integral element of $\mathcal{I}^n_{m,m-1}$. Since the exterior ideal does not contain any functions, every point of $\mathbf{\Sigma}^n_{m,m-1}$ is a $0$-integral element. Let $(E_{0})_{z} = z\in \mathbf{\Sigma}^n_{m,m-1}$.  A vector $\xi$ in the tangent space of $\mathbf{\Sigma}^n_{m,m-1}$ is of the form
\begin{equation}
\xi = \xi^\lambda_\mathcal{M} X_{\lambda}+ \xi^{A}Y_{A}.
\end{equation}
By considering the polar space, we obtain Cartan characters as previously. We then choose  the integral element in the following way:
\begin{equation}
e_{\lambda}=X_{\lambda}+ H^{a}_{i\lambda} Y_{\sigma(^{a}_{i})},
\end{equation}
where the coefficients $H^{a}_{i\lambda}$ are provided by the lemma \ref{LemFond}, which assures the existence of solutions to the successive polar systems during the construction of the integral flag.  The coefficients $\xi^{\sigma(^{i}_{j})}$ for all $1\leqslant i < j \leqslant n$ vanish for all the vectors $e$ because of $\varpi^{\sigma(^{i}_{j})}$. Let us  denote $E_{\lambda}=\mathrm{span}\{ e_{1},\dots ,e_{\lambda}\}$. The integral flag is then $F=E_{0}\subset E_{1}\subset \dots \subset E_{m-1}\subset E_{m}$. Cartan characters are the same as computed previously and the Cartan test assures that the flag is ordinary. By construction, the flag does not annihilate the volume form $\eta^\Lambda$.

\subsection{Proof of lemma \ref{LemFond}}

The generalized Gauss map $\mathcal{G}^n_{m,m-1}$ defined on $\mathcal{W}^n_{m,m-1}\otimes\mathbb{R}^n\otimes\mathbb{R}^m$ with values in $\mathcal{K}^n_{m,m-1}$ is a submersion if and only if the differential $\mathrm{d}\mathcal{G}^n_{m,m-1} \in \mathcal{L}(\mathcal{W}^n_{m,m-1}\otimes\mathbb{R}^n\otimes\mathbb{R}^m;\mathcal{K}^n_{m,m-1} )$, which  has  $m(m-1)n(n-1)/4$ lines and $\kappa^n_{m,m-1}\times m\times n$ columns, is of maximal rank.  
 \\

In what follows, we make the assumption that  $\psi^1_{\Lambda \smallsetminus m} =1$ and $\psi^2_{\Lambda \smallsetminus m} =\dots  = \psi^n_{\Lambda \smallsetminus m} =0$. It is always possible  by changing the notation and reindexing.   With this assumption, the generalized Cartan identity shows that for all $a$, the coefficient $H^{a}_{1m}$ on a given point of the manifold,  is a linear combination of the $H_{i\lambda}$ where $\lambda  \neq m$.  When $n = m =2$, we assume that the determinant  $\mathrm{det}\psi = (\psi^1_{1}\psi^2_{2} - \psi^1_{2}\psi^2_{1})\neq 0$. In order to understand the proof of the submersitivity of $\mathcal{G}^{n}_{m,m-1}$, we first  explain and show  the proof for two special cases: when the vector bundle is of rank 3  over a manifold of dimension 2,  and when the rank of the vector bundle is arbitrary ($n\geqslant 2$) over a manifold of dimension 2. The proof of the  surjectivity of the generalized Gauss map is established afterwards.

\subsubsection{Submersitivity of the generalized Gauss map}
We will proceed step by step in order to expound the proof of Lemma \ref{LemFond}: For a warm-up, we start with the case $(\VV^3, \Mm^2 , g, \nabla, \phi)_{1}$, then the case of a general vector bundle over a surface, i.e., $(\VV^n, \Mm^2, g , \nabla)_{1}$,  next, the case of a vector bundle of rank 2 over an $m$-dimensional manifold, i.e., $(\VV^2, \Mm^m, g , \nabla, \phi)_{m-1}$, and  finally, we expound the conservation laws case, i.e.,  $(\VV^n, \Mm^m, g , \nabla, \phi)_{m-1}$.

Recall that the generalized Gauss map associates $H=(H^{a}_{i\lambda})$ with $ \Big((\Gg^n_{m,m-1})^{i}_{j;\lambda\mu}\Big) = (H_{i\lambda}H_{j\mu} - H_{i\mu}H_{j\lambda})^{i}_{j;\lambda\mu}$. The differential of $\Gg^n_{m,m-1}$ is then:
\begin{equation}
\mathrm{d}\Gg^n_{m,m-1} = \frac{\partial \Gg^n_{m,m-1}}{\partial H^{a}_{i \lambda}}\mathrm{d}H^{a}_{i\lambda}
\end{equation}
where
\begin{equation}
\mathrm{d}(\Gg^n_{m,m-1})^{i}_{j;\lambda\mu} = H_{j\mu}\mathrm{d}H_{i\lambda}+ H_{i\lambda}\mathrm{d}H_{j\mu} - H_{j\lambda}\mathrm{d}H_{i\mu}  - H_{i\mu}\mathrm{d}H_{j\lambda}.
\end{equation}

Denote by $\epsilon^{i}_{j;\lambda\mu}$ the natural basis on $\Kk^n_{m, m-1} = \wedge^2 \RR^n\otimes \wedge^2 \RR^m$.

\paragraph{\textsc{The case} $\mathbf{(\VV^3, \Mm^2 , g, \nabla , \phi)_{1}}$} Consider a vector bundle $\VV^3$ of rank 3 over a 2-dimensional differentiable manifold $\Mm^2$, endowed with a metric $g$ and a connection $\nabla$ compatible with $g$. Let $\phi$ be a non-vanishing  covariantly closed $\VV^2$-valued differential 1-form. By assumption,
\begin{equation}
\phi = E_{i}\phi = E_{i}\psi^{i}_{\lambda}\eta^\lambda = \left(\begin{array}{cc}1 & \psi^1_2 \\0 & \psi^2_2 \\0 & \psi^3_{2}\end{array}\right)\wedge\left(\begin{array}{c}\eta^1 \\\eta^2\end{array}\right).
\end{equation}

The generalized Cartan identities for each normal direction $a$ are:
\begin{equation}
H_{12}^{a} = \psi^1_2 H^a_{11} + \psi^2_{2}H^a_{21}+\psi^3_{2}H^a_{31}.
\end{equation}

The curvature tensors' space is $\Kk^3_{2,1} =\wedge^2 \RR^3 \otimes \wedge^2 \RR^2 = \wedge^2\RR^3 \otimes \RR = \text{span}\{\epsilon^1_{2;12}, \epsilon^1_{3;12}, \epsilon^2_{3;12},  \}$.

The generalized Gauss equations are:
\begin{equation}
\left\{\begin{array}{ccc} H_{11}.H_{22} - H_{12}.H_{21} & = & \Rr^1_{2;12} \\ H_{11}.H_{32} - H_{12}.H_{31} & = & \Rr^1_{3;12} \\ H_{21}.H_{32} - H_{22}.H_{31} & = & \Rr^2_{3:12}\end{array}\right.
\end{equation}


Taken into consideration the generalized  Cartan identities, the differential of the generalized Gauss map is:
\begin{equation}\nonumber
\mathrm{d}\Gg^3_{2,1} \hspace{-1mm}=\hspace{-1mm}\left(\begin{array}{c} \mathrm{d}(\Gg^3_{2,1})^1_{2,12} \\  \mathrm{d}(\Gg^3_{2,1})^1_{3,12}  \\  \mathrm{d}(\Gg^3_{2,1})^2_{3,12}  \end{array}\right)\hspace{-1mm}=\hspace{-1mm}\left(\begin{array}{ccccc}H_{22} & -\psi^i_2 H_{i1}  & 0  & H_{11} & 0 \\H_{32} & 0 & -\psi^i_2 H_{i1}  & 0 & H_{11} \\0 & H_{32} & -H_{22}& -H_{31} & H_{21}\end{array}\right)\hspace{-1mm}.\hspace{-1mm} \left(\begin{array}{c} \mathrm{d}H_{11} \\  \mathrm{d}H_{21}  \\  \mathrm{d}H_{31}  \\  \mathrm{d}H_{22} \\  \mathrm{d}H_{32} \end{array}\right)
\end{equation}

Note that $H_{i\lambda}$ are vectors in the Euclidean space $\Ww^3_{2,1}$ of dimension $\kappa^3_{2,1}$  which  must be determined. We want  to extract from the $\Ww^3_{2,1}$-valued matrix $\mathrm{d}\Gg^3_{2}$ a submatrix of maximal rank (rank 3).  Denote by $L$ the subspace of cotangent of $\Ww^3_{2,1}\RR^3\otimes \RR^2$ defined by $\mathrm{d}H_{11} = \mathrm{d}H_{21} =\mathrm{d}H_{31}=0$. Then\footnote{$\mathrm{d}\Gg^3_{2,1}|_{L}$ is the submatrix of $\mathrm{d}\Gg^3_{2, 1}$ defined by: $((\mathrm{d}\Gg^3_{2}(\partial / \partial H_{22}))_{a},(\mathrm{d}\Gg^3_{2}(\partial / \partial H_{23}))_{a}  )$.} $\mathrm{d}\Gg^3_{2,1}|_{L}$ is :
\begin{equation}
\mathrm{d}\Gg^3_{2}|_{L} = \left(\begin{array}{cc}H_{11} & 0 \\0 & H_{11} \\-H_{31} & H_{21}\end{array}\right). \left(\begin{array}{c}\mathrm{d}H_{22}\\ \mathrm{d}H_{32}\end{array}\right).
\end{equation}

Therefore, if $\kappa^3_{2,1} \geqslant 2$, the matrix $\mathrm{d}\Gg^3_{2}|_{L}$ is of maximal rank if $H_{11}$ and $H_{21}$ are linearly independent vectors of $\Ww^3_{2,1}$. For instance, if $\kappa^3_{2,1} = 2$, i.e., the normal directions are $a=4, 5$, then
\begin{equation}
\mathrm{d}\Gg^3_{2}|_{L} = \left(\begin{array}{cccc}H^4_{11} & H^5_{11} &  0 & 0 \\0 & 0 & H^4_{11} & H^5_{11} \\-H^4_{31} & -H^5_{31} & H^4_{21} &H^5_{21}\end{array}\right). \left(\begin{array}{c}\mathrm{d}H^4_{22}\\ \mathrm{d}H^5_{22} \\  \mathrm{d}H^4_{32} \\ \mathrm{d}H^5_{32}\end{array}\right).
\end{equation}
is of maximal rank if $H_{11}$ and $H_{21}$ are linearly independent vectors. \\

Before investigating the submersitivity of the genralized Gauss map, let us first define a flag of the subspaces of $\Kk^n_{m,m-1}$.

\paragraph{Flag of $\mathbf{\Kk^n_{m,m-1}}$:}
Let us define the following subspaces of $\mathcal{K}^n_{m,m-1}$ as follows: for  $k=2, \dots, n$
\begin{equation}\nonumber
 \Ee^k|^n_{m,m-1} = \{Ê(\Rr^{i}_{j;\lambda \mu}) \in \Kk^n_{m,m-1} | \Rr^{i}_{j; \lambda \mu } = 0,  \text{if } 1\leqslant i< j\leqslant k \text{ and }\forall 1\leqslant \lambda < \mu \leqslant m \} 
 \end{equation}
 and for $\nu = 2, \dots, m$
 \begin{equation}\nonumber
  \quad  \Ee_{\nu}|^n_{m,m-1}  = \{Ê(\Rr^{i}_{j;\lambda \mu}) \in \Kk^n_{m,m-1} | \Rr^{i}_{j; \lambda \mu } = 0, \text{if }1\leqslant \lambda < \mu \leqslant \nu \text{ and }\forall 1 \leqslant i < j \leqslant n \}.
\end{equation}
By convention, $\Ee^1|^n_{m,m-1} =\Ee_{1}|^n_{m,m-1} = \Kk^n_{m,m-1}$. 
Therefore, 
\begin{eqnarray*}
0 = \Ee^n|^n_{m,m-1} \subset \Ee^{n-1}|^n_{m,m-1}  \subset \Ee^{n-2}|^n_{m,m-1} \subset \dots \subset \Ee^2|^n_{m,m-1}  \subset  \Kk^n_{m,m-1}\label{InclusionEek} \\
0 = \Ee_m|^n_{m,m-1} \subset \Ee_{m-1}|^n_{m,m-1}  \subset \Ee_{m-2}|^n_{m,m-1} \subset \dots \subset \Ee_{2}|^n_{m,m-1}  \subset \Kk^n_{m,m-1} .\label{InclusionEeNu}
\end{eqnarray*}

\begin{Ex}[$\mathbf{(\VV^3 , \Mm^4 , g , \nabla , \phi)_{3}}$]
An element in  $\Kk^3_{4,3} = \wedge^2\RR^3\otimes \wedge^2\RR^4  \simeq \RR^{18}$  is:

\begin{equation}
\Rr = \left(\begin{array}{cccccc} \Rr^1_{2;12} &  \Rr^1_{2;13}  &  \Rr^1_{2;23}  &  \Rr^1_{2;14}  &  \Rr^1_{2;24}  &  \Rr^1_{2;34}  \\  \Rr^1_{3;12} &  \Rr^1_{3;13}  &  \Rr^1_{3;23}  &  \Rr^1_{3;14}  &  \Rr^1_{3;24}  &  \Rr^1_{3;34}  \\ \Rr^2_{3;12}  &  \Rr^2_{3;13}  &  \Rr^2_{3;23}  &  \Rr^2_{3;14}  &  \Rr^2_{3;24}  &  \Rr^2_{3;34}  \end{array}\right)
\end{equation}
and  if $\Rr$ is in $\Ee^2|^3_{4,3} $ and in  $\Ee^3|^3_{4,3} $ then respectively
\begin{equation}
\Rr=\left(\begin{array}{cccccc}0 &  0  &0 & 0  & 0  & 0 \\ \ast&  \ast &  \ast  &  \ast  & \ast  &  \ast  \\ \ast &  \ast  &  \ast  &  \ast  &  \ast  &  \ast  \end{array}\right) \text{ and } \Rr=(0)
\end{equation}
and if $\Rr$ is in $\Ee_{2}|^3_{4,3}$, $\Ee_{3}|^3_{4,3}$ and in $\Ee_{4}|^3_{4,3}$ then  respectively
\begin{equation*}
\Rr=\left(\begin{array}{cccccc} 0 &  \ast & \ast & \ast  & \ast & \ast  \\  0 &  \ast & \ast & \ast & \ast & \ast \\ 0  & \ast & \ast &  \ast & \ast  &  \ast  \end{array}\right), \Rr=\left(\begin{array}{cccccc}0 &  0  & 0  &  \ast& \ast  &\ast  \\ 0 & 0  & 0  & \ast  &  \ast  &  \ast  \\ 0  &  0  & 0  &  \ast  &  \ast  & \ast  \end{array}\right), \text{ and }\quad  \Rr=0
\end{equation*}
\end{Ex}

\paragraph{\textsc{The case} $\mathbf{(\VV^n, \Mm^2 , g, \nabla , \phi)_{1}}$} Recall that $\Kk^n_{2,1} = \wedge^2\RR^n\otimes\RR$. Some columns in the Jacobian of $\Gg^n_{2,1}$ are expressed as follows:  for $k=2, \dots , n,$
\begin{equation}\label{DiffGaussMapm=2}
\mathrm{d}\mathcal{G}^n_{2,1} \Big(\frac{\partial}{\partial H^{a}_{k2}}\Big) = \Big(\sum_{i=1}^{k-1}H^{a}_{i1}\epsilon^{i}_{k;12} +  (\text{terms in } \Ee^k|^n_{2,1}) \Big)\in \mathcal{E}^{k-1}|^n_{2,1} .
\end{equation}
Note that $\Ee^{n}|^n_{2,1} = 0$, and hence, 
\begin{equation}\label{dGm2}
\mathrm{d}\mathcal{G}^n_{2,1} (\partial/ \partial H^{a}_{n2}) = \Big(\sum_{i=1}^{n-1}H^{a}_{i1} \epsilon^{i}_{n;12} \Big) \in \mathcal{E}^{n-1}|^n_{2,1}.
\end{equation}

From the linear map $\mathrm{d}\mathcal{G}^n_{2,1}$, we want to extract a submatrix of maximal rank.  Consider the submatrix $\Big( (\mathrm{d}\mathcal{G}^{n}_{2,1}(\partial/\partial H^{a}_{22}))_{a}, \dots , (\mathrm{d}\mathcal{G}^{n}_{2,1}(\partial/ \partial H^{a}_{n2} ))_{a}\Big)$. Each term $(\mathrm{d}\mathcal{G}^{n}_{2,1}(\partial/ \partial H^{a}_{k2} ))_{a}$, for a fixed $k$, is a matrix with $n(n-1)/2$ lines and $\kappa^n_{2,1}$ columns.  The equations (\ref{DiffGaussMapm=2}), (\ref{dGm2}) and the inclusions (\ref{InclusionEek}) show that the submatrix $\Big( (\mathrm{d}\mathcal{G}^{n}_{2,1}(\partial/\partial H^{a}_{22}))_{a}, \dots , (\mathrm{d}\mathcal{G}^{n}_{2,1}(\partial/ \partial H^{a}_{n2} ))_{a}\Big)$ is of maximal rank if the vectors $H_{11}, H_{21}, \dots H_{(n-1)1}$ are linearly independent vectors of $\mathcal{W}^n_{2,1}$ and $\kappa^n_{2,1} \geqslant(n-1)$ where the minimal embedding codimension $\kappa^n_{2,1}$ is given by the dimension of $\mathcal{E}^{n-1}|^n_{2,1}$.  Indeed, the matrix $\Big( (\mathrm{d}\mathcal{G}^{n}_{2,1}(\partial/\partial H^{a}_{22}))_{a}, \dots , (\mathrm{d}\mathcal{G}^{n}_{2,1}(\partial/ \partial H^{a}_{n2} ))_{a}\Big)$ is  triangular by different sized blocks. This is due to the inclusions (\ref{InclusionEek}) of the spaces $\Ee^k|^n_{2,1}$ . Note that the matrix $\Big( (\mathrm{d}\mathcal{G}^{n}_{2,1}(\partial/\partial H^{a}_{22}))_{a}, \dots , (\mathrm{d}\mathcal{G}^{n}_{2,1}(\partial/ \partial H^{a}_{n2} ))_{a}\Big)$ is rectangular, i.e., $n(n-1)/2$ lines and  $(\kappa^n_{2,1}\times (n-1))$ columns. There are actually $(n-1)$ terms in the ''diagonal'' and they all have the  same number of columns $\kappa^n_{2,1}$. The first term of the ''diagonal'' has one line and obviously starts at the first line, the second term has 2 lines and is at the second line, the third term has 3 lines and starts at the line number 1+2 = 3, \dots , and the last term has $(n-1)$ lines and starts at the line number $(n-2)(n-1)/2$.     From (\ref{DiffGaussMapm=2}) and  (\ref{dGm2}), the ''diagonal'' of $\Big( (\mathrm{d}\mathcal{G}^{n}_{2,1}(\partial/\partial H^{a}_{22}))_{a}, \dots , (\mathrm{d}\mathcal{G}^{n}_{2,1}(\partial/ \partial H^{a}_{n2} ))_{a}\Big)$ is: $\mathrm{diag}\Big((H^{a}_{11})_{a}, ^t(H^{a}_{11},H^{a}_{21})_{a}, \dots , ^t(H^{a}_{11}, \dots , H^{a}_{(n-1)1})_{a}\Big)$, and since $0\subset \Ee^{n-1}\subset \Ee^{n-2} \subset\dots\subset \Ee^2\subset \Ee^1 = \Kk^n_{2,1}$, the terms above this ''diagonal'' vanish in the matrix $\Big( (\mathrm{d}\mathcal{G}^{n}_{2,1}(\partial/\partial H^{a}_{22}))_{a}, \dots , \linebreak(\mathrm{d}\mathcal{G}^{n}_{2,1}(\partial/ \partial H^{a}_{n2} ))_{a}\Big)$.  Note that $^t(H_{11}, \dots , H_{k1})_{a}$ is a matrix with $k$ lines and $\kappa^n_{2,1}$ columns. The condition of being linearly independent for  the vector $(H_{11}, \dots H_{(n-1)1})$ assures that one can always extract, for each term of the diagonal, a submatrix of maximal rank. For instance, the ''diagonal'' term of $\mathrm{d}\Gg^n_{2,1}(\partial /\partial H^{a}_{32})$ is $^t(H^a_{11}, H^{a}_{21})$, which is a $2\times \kappa^n_{2,1}$ matrix, and since the two vectors are linearly independent, there exists an invertible  $2\times2$ submatrix.   The same argument holds for each term of the ''diagonal'', and finally,  $\kappa^n_{2,1} \geqslant \mathrm{dim}(\Ee^{n-1}|^n_{2,1})$ assures that the last terms of the ''diagonal'', $(\mathrm{d}\Gg^n_{2,1}(\partial /\partial H^a_{n2}))_{a}$, are of maximal rank. \\

\paragraph{\textsc{The case} $\mathbf{(\VV^n, \Mm^m , g, \nabla , \phi)_{m-1}}:$} For the conservation laws case, we  define the following subspaces of $\Kk^n_{m,m-1}$: for  $k=2, \dots, n$  and  for $ \nu = 2, \dots, m$,

\begin{equation*}
\Ee_{\nu}^k |^n_{m,m-1} = \{(\Rr)^{i}_{j;\lambda\mu}\in\Kk^n_{m,m-1}| \Rr^{i}_{j;\lambda\mu} = 0, \text{ if } 1\leqslant i< j\leqslant k \text{ and } 1\leqslant \lambda < \mu \leqslant \nu  \}
\end{equation*}
and hence,  $\Ee^n_{\nu}|^n_{m,m-1} = \Ee_{\nu}|^n_{m,m-1} \quad \text{ and } \quad \Ee^k_{m}|^n_{m,m-1} = \Ee^k|^n_{m,m-1}.$
By convention, $\Ee^1_{\nu}|^n_{m,m-1} = \Kk^n_{m,m-1}$ and $\Ee^k_{1}|^n_{m,m-1} = \Kk|^n_{m,m-1}$.

\begin{Rem} Let us  fix $\nu$ and $k$.  We have the same kind of flags as in (\ref{InclusionEek}) and (\ref{InclusionEeNu}):
\begin{eqnarray*}
 \Ee_{\nu}|^n_{m,m-1} \subset \Ee^{n-1}_{\nu}|^n_{m,m-1}  \subset \Ee^{n-2}_{\nu}|^n_{m,m-1} \subset \dots \subset \Ee^2_{\nu}|^n_{m,m-1}  \subset \Kk^n_{m,m-1} \\
\Ee^{k}|^n_{m,m-1} \subset \Ee^k_{m-1}|^n_{m,m-1}  \subset \Ee^k_{m-2}|^n_{m,m-1} \subset \dots \subset \Ee^k_{2}|^n_{m,m-1}  \subset  \Kk^n_{m,m-1}.
\end{eqnarray*}
\end{Rem}

\begin{Ex}[$\mathbf{(\VV^3 , \Mm^4 , g , \nabla , \phi)_{3}}$-Continued]
$\Ee^2_{4}|^3_{4,3} = \Ee^2|^3_{4,3}$, $\Ee^3_{2}|^3_{4,3} = \Ee_{2}|^3_{4,3}$,$\Ee^3_{3}|^3_{4,3} = \Ee_{3}|^3_{4,3}$ and $\Ee^3_{4}|^3_{4,3} = 0$ and if $\Rr$ is in $\Ee^{2}_{2}|^3_{4,3}$, $\Ee^{2}_{3}|^3_{4,3}$, then respectively
\begin{equation}
\Rr= \left(\begin{array}{cccccc} 0 &  \ast  &  \ast  &  \ast &  \ast &  \ast \\  \ast & \ast & \ast  &  \ast  & \ast  &  \ast  \\ \ast &  \ast &  \ast &  \ast &  \ast  &  \ast  \end{array}\right), 
\Rr=\left(\begin{array}{cccccc}0 & 0 &  0  & \ast  & \ast & \ast \\  \ast &  \ast  & \ast& \ast & \ast&  \ast \\ \ast &  \ast  & \ast &  \ast &  \ast&  \ast  \end{array}\right).
\end{equation}
\end{Ex}

\begin{Prop}[Extension of (\ref{InclusionEeNu})]\label{ExtensionInclusion}For $(\VV^n , \Mm^m, g, \nabla, \phi)_{m-1}$, we can have a longer flag by replacing in (\ref{InclusionEeNu}) each inclusion of the type  $\Ee_{\nu}|^n_{m,m-1} \subset \Ee_{(\nu-1)}|^n_{m,m-1}$, for $\nu = 2,\dots , m$,  by 
\begin{equation}\label{IncrustFlag}
\mathbf{\Ee_{\nu}} \subset  \Big(\Ee_{(\nu - 1)}\cap \Ee^{n-1}_{\nu}\Big)  \subset  \Big(\Ee_{(\nu - 1)}\cap \Ee^{n-2}_{\nu}\Big)\subset \dots \subset  \Big(\Ee_{(\nu - 1) }\cap \Ee^{3}_{\nu}\Big)  \subset \Big( \Ee_{(\nu - 1) }\cap \Ee^{2}_{\nu}\Big)\subset \mathbf{\Ee_{(\nu-1)}}.
\end{equation}
Note that we dropped $|^n_{m,m-1}$ for each subspace $\Ee$, in the above equation, for more clarity.  
\end{Prop}

\begin{Ex}[$\mathbf{(\VV^4 , \Mm^5 , g , \nabla , \phi)_{4}}$] We drop in this example the signs $|^4_{5,4}$ next to the subspaces $\Ee^k_{\nu}|^4_{5,4}$. When we put (\ref{IncrustFlag}) in (\ref{InclusionEeNu}), we obtain
$0=\Ee_{5}\subset  \Big(\Ee_{4}\cap \Ee^3_5\Big) \subset  \Big(\Ee_{4}\cap \Ee^2_5\Big) \subset\Ee_{4} \subset \Big(\Ee_{3}\cap \Ee^3_4\Big)\subset  \Big(\Ee_{3}\cap \Ee^2_4\Big)  \subset \Ee_{3}\subset \Big(\Ee_2\cap \Ee^3_3\Big)\subset  \Big(\Ee_{2}\cap \Ee^2_3\Big)  \subset \Ee_{2}\subset  \Ee^3_{2}\subset \Ee^2_{2} \subset \Ee_{1}= \Kk^4_{5,4}$.
\end{Ex}



We proceed in the same way to prove Lemma \ref{LemFond}. The inclusion of the spaces $\Ee^k_{\nu}|^n_{m,m-1}$ is more complex and is given by the Proposition \ref{ExtensionInclusion}. We have, for  $k=2,\dots, n$ and  $\nu=2,\dots , m$  
\begin{equation}\label{DiffGaussMap}
\mathrm{d}\mathcal{G}^n_{m,m-1} (\partial/ \partial H^{a}_{k\nu}) = \Big(\hspace{-0.8cm} \sum_{\tiny{\left.\begin{array}{c}i=1,\dots , k-1 \\ \lambda = 1, \dots \nu-1 \end{array}\right.}}\hspace{-0.8cm}H^{a}_{i\lambda}\epsilon^{i}_{k;\lambda\nu} 
 + (\text{terms in } \Ee^{k+1}_{\nu - 1})\Big)\in \mathcal{E}^{k-1}_{\nu-1}|^n_{m,m-1}
\end{equation}
and since $\Ee^n_{k}|^n_{m,m-1} = 0$,
\begin{equation}
\mathrm{d}\mathcal{G}^n_{m,m-1} (\partial/ \partial H^{a}_{nm}) = \Big(\hspace{-0.8cm} \sum_{\tiny{\left.\begin{array}{c}i=1,\dots , n-1 \\ \lambda = 1, \dots m-1 \end{array}\right.}}\hspace{-0.8cm}H^{a}_{i\lambda}\epsilon^{i}_{n;\lambda m} \Big)\in \mathcal{E}^{n-1}_{m-1}|^n_{m,m-1}.
\end{equation}

As we explained previously, from the linear map $\mathrm{d}\mathcal{G}^n_{m,m-1}$, we want to extract a submatrix of maximal rank. Consider the submatrix $\Big((\mathrm{d}\mathcal{G}^{n}_{m,m-1}(\partial / \partial H^{a}_{22}))_{a} , \dots,\linebreak (\mathrm{d}\mathcal{G}^{n}_{m,m-1}(\partial / \partial H^{a}_{n2}))_{a} , \dots ,(\mathrm{d}\mathcal{G}^{n}_{m,m-1}(\partial / \partial H^{a}_{2m}))_{a}, \dots , (\mathrm{d}\mathcal{G}^{n}_{m,m-1}(\partial / \partial H^{a}_{nm}))_{a}\Big)$ which has $n(n-1)m(m-1)/4$ lines and $\kappa^n_{m,m-1}\times (n-1)(m-1)$ columns. This matrix is of maximal rank if the vectors $(H_{i\lambda})_{i=1,\dots , (n-1)\text{ and }\lambda = 1 , \dots , m-1}$ are linearly independent vectors of $\mathcal{W}^n_{m,m-1}$ where $\kappa^n_{m,m-1} \geqslant (n-1)(m-1).$ The minimal embedding codimension is given by the dimension of $(\Ee^{n-1}\cap\mathcal{E}_{m-1}|^n_{m,m-1})$. Indeed, Proposition (\ref{ExtensionInclusion}) shows that the submatrix is triangular by different sized blocks and that the terms above the block-diagonal are zero. There are $(n-1)(m-1)$ terms in the ''diagonal'' and they have the same number of columns $\kappa^n_{m-1}$.  



\subsubsection{The surjectivity of the generalized Gauss map}
It remains to show that the generalized Gauss map is surjective, namely
\begin{equation}
\Gg^n_{m,m-1}(\Hh^n_{m,m-1}) = \Kk^n_{m,m-1}.
\end{equation}

It is sufficient to show that there exists a pre-image of $0$, i.e., vectors $H_{i\lambda}$ in $\Ww^n_{m,m-1}$, satisfying generalized Cartan identities and such that the set $\{H_{i\lambda} \}$ for $i=1,\dots , n-1$ and $\lambda = 1, \dots , m-1$ are linearly independent vectors in $\Ww^n_{m,m-1}$. Indeed, the differential of the generalized Gauss map being surjective implies that $\Gg^n_{m,m-1}(\Hh^n_{m,m-1})$ will contain a neighborhood  of 0 in $\Kk^n_{m,m-1}$, and thus $\Gg^n_{m,m-1}(\Hh^n_{m,m-1}) = \Kk^n_{m,m-1}$ as $\Gg^n_{m,m-1}(\rho H) = \rho^2 \Gg^n_{m,m-1}(H)$.\\

We will construct a pre-image of $0$ in $\Hh^n_{m,m-1}$. Recall that $\Ww^n_{m,m-1}$ is of dimension $\kappa^n_{m,m-1} \geqslant (n-1)(m-1)$. We can choose $H_{i\lambda}$ as follows:
\begin{equation}
\{H_{i\lambda} \}_{i=1,\dots, n-1 \text{ and } \lambda = 1 , \dots , m-1} \text{ is an orthonormal set of vectors in }\Ww^n_{m,m-1}
\end{equation}

\begin{equation}
H_{n1} = H_{n2}=\dots = H_{nm}=0
\end{equation}

\begin{equation}\label{Him}
\text{For } j=2, \dots, m, \qquad H_{jm} =\sum_{\tiny{\left.\begin{array}{c}i=1,\dots, n-1 \\\lambda = 1, \dots , m-1\end{array}\right.}} A^{i\lambda}_{j}H_{i\lambda}
\end{equation}
where 
\begin{equation}
A^{1\lambda}_{j} = \psi^j_{\Lambda\smallsetminus \lambda} \qquad \text{and} \qquad A^{i\lambda}_{j} = A^{j\lambda}_{i} .
\end{equation}



\section{Conservation laws for covariant divergence free energy-momentum tensors}We present here an application for our main result to covariant divergence free energy-momentum tensors. 

\begin{Prop}
Let $(\mathcal{M}^{m}, g)$ be a $m$-dimensional real analytic Riemannian manifold, $\nabla$ be the Levi-Civita connection and $T$ be contravariant 2-tensor with a vanishing covariant divergence. There then  exists a conservation law for $T$ on $\mathcal{M}\times\mathbb{R}^{m+(m-1)^2}$.
\end{Prop}

\begin{proof}
Let us consider a bivector $T\in \Gamma(\mathrm{T}\mathcal{M}\otimes \mathrm{T}\mathcal{M})$ which  is expressed in a chart by $T=T^{\lambda\mu}\xi_{\lambda}\otimes\xi_{\mu}$, where $(\xi_{1} , \dots , \xi_{m})$ is the dual basis of an orthonormal moving coframe $(\eta^1 , \dots , \eta^m)$. The volume form  is denoted by $\eta^{\Lambda} = \eta^1\wedge\dots \wedge\eta^m$. Using the interior product, we can associate  any bivector $T$ with a $\mathrm{T}\mathcal{M}$-valued $m$-differential form $\tau$ defined as follows:

\begin{equation*}
\begin{split}
\Gamma ( \mathrm{T}\mathcal{M}\otimes \mathrm{T}\mathcal{M}) &\longrightarrow \Gamma (\mathrm{T}\mathcal{M}\otimes \wedge^{(m-1)}\mathrm{T}^\ast\mathcal{M})\\
T=T^{\lambda\mu}\xi_{\lambda}\otimes\xi_{\mu} & \longmapsto \tau = \xi_{\lambda}\otimes \tau^{\lambda}=\xi_{\lambda}\otimes \Big(T^{\lambda\mu}(\xi_{\mu} \lrcorner \eta^{\Lambda})\Big).
\end{split}
\end{equation*}

The tangent space $\mathrm{T}\mathcal{M}$ is endowed with the Levi-Civita connection $\nabla$. Let us compute the covariant derivative of $\tau$. 
\begin{equation}
\mathrm{d}_{\nabla}\tau = \xi_{\lambda}\otimes (\mathrm{d}\tau^{\lambda} + \eta^{\lambda}_{\mu}\wedge\tau^\mu)
\end{equation}

 On one hand, using the first Cartan equation 
 that expresses the vanishing of the torsion of the Levi-Civita connection and the expression of the Christoffel symbols 
  in terms of the connection 1-form, we obtain
 \begin{equation}
\begin{split}
\mathrm{d}\tau^\lambda &= \mathrm{d}\Big(T^{\lambda\mu}(\xi_{\mu}\lrcorner \eta^{\Lambda})\Big) = \mathrm{d}(T^{\lambda\mu})\wedge(\xi_{\mu}\lrcorner \eta^{\Lambda})+ T^{\lambda\mu}\mathrm{d}(\xi_{\mu}\lrcorner \eta^{\Lambda})
= \Big(\xi_{\mu}(T^{\lambda\mu}) + T^{\lambda\mu}\Gamma^{\nu}_{\nu\mu}\Big)\eta^{\Lambda} 
\end{split}
\end{equation}
and
\begin{equation}
\eta^{\lambda}_{\mu}\wedge\tau^{\mu} = \eta^{\lambda}_{\mu}\wedge T_{\mu\nu}(\xi_{\nu}\lrcorner\eta^{\Lambda})= \Big( T^{\mu\nu}\Gamma^{\lambda}_{\nu\mu}\Big)\eta^{\Lambda}
\end{equation}
consequently
\begin{equation}
\mathrm{d}_{\nabla}\tau = \xi_{\lambda}\otimes\Big[\Big(\xi_{\mu}(T^{\lambda\mu}) + T^{\lambda\mu}\Gamma^{\nu}_{\nu\mu}+ T^{\mu\nu}\Gamma^{\lambda}_{\nu\mu}\Big)\eta^{\Lambda} \Big].
\end{equation}

On the other hand, a straightforward computation of the divergence of the bivector leads us to
\begin{equation}
\nabla_{\mu}T^{\lambda\mu} = \xi_{\mu}(T^{\lambda\mu}) + T^{\lambda\mu}\Gamma^{\nu}_{\nu\mu}+ T^{\mu\nu}\Gamma^{\lambda}_{\nu\mu} \quad \text{ for all }\lambda=1,\dots , m.
\end{equation}
 We conclude then that 
 \begin{equation}
\mathrm{d}_{\nabla}\tau = 0 \Leftrightarrow \nabla_{\mu}T^{\lambda\mu}=0 \quad \forall \lambda=1,\dots, m.
\end{equation}
Hence, for an $m$-dimensional Riemannian manifold $\mathcal{M}$, the main result of this article assures the existence of an isometric embedding $\Psi: T\mathcal{M} \longrightarrow \mathcal{M}\times \mathbb{R}^{m+(m-1)^2}$ such that $\mathrm{d}(\Psi (\tau))=0$ is a conservation law for a  covariant divergence free energy-momentum tensor.
\end{proof}

 For instance, if $\mathrm{dim}\mathcal{M} = 4$, then $\Psi(\tau)$ is a closed differential $3$-form on $\mathcal{M}$ with values in $\mathbb{R}^{13}$.

\section*{acknowledgements}
The author is very grateful to Sara Carey for  reading the manuscript  and to Fr\'ed\'eric H\'elein for his helpful remarks and suggestions.

\bibliographystyle{amsalpha}
\bibliography{Biblio}

\end{document}